\documentclass[11pt]{article}

\usepackage{amsfonts}
\usepackage{graphics}
\usepackage{amssymb}

\newtheorem{theorem}{Theorem}[section]
\newtheorem{lemma}{Lemma}[section]
\newtheorem{definition}{Definition}[section]
\newtheorem{proposition}[theorem]{Proposition}
\newtheorem{corollary}[theorem]{Corollary}
\newtheorem{example}{Example}[section]

\newcommand{\eop }{ \hfill $\Box$ }

\newcommand{\paragrafo}{\vrule height 0pt width 0pt depth 0pt\hbox to\parindent{\hss}}

\begin{document}

\begin{center}
{\Large A Pluzhnikov's Theorem, Brownian motions and Martingales in Lie Group with skew-symmetric connections}

\end{center}

\vspace{0.3cm}

\begin{center}
{\large Sim\~ao Stelmastchuk }\\

\textit{Universidade Estadual de Campinas, 13.081-970 -  Campinas - SP, Brazil. e-mail: simnaos@gmail.com}
\end{center}

\vspace{0.3cm}

\begin{abstract}
Let $G$ be a Lie Group with a left invariant connection such that its connection function is skew-symmetric. Our main goal is to show a version of Pluzhnikov's Theorem for this kind of connection. To this end, we use the stochastic logarithm. More exactly, the stochastic logarithm gives characterizations for Brownian motions and Martingales in $G$, and these characterzations are used to prove Pluzhnikov's Theorem.
\end{abstract}


\vspace{0.3cm} \noindent {\bf MSC2010 subject classification:}
 53C43, 58E20,  58J65, 60H30, 60G48

\section{Introduction}

Let $G$ be a Lie group, $\mathfrak{g}$ its lie algebra and $\omega_{G}$ the Maurer-Cartan form on $G$. K. Nomizu, in \cite{nomizu}, has proved that  there is an one-to-one association between left invariant connection in $G$ and bilinear applications $\alpha$ from $\mathfrak{g}\otimes \mathfrak{g}$ into $\mathfrak{g}$, which is called connection function. In this work, we are only interested in  skew-symmetric connection functions.

Our main goal is to prove a version of the following Theorem, in the Riemannian case, that was proved by A.I. Pluzhnikov in \cite{pluzhnikov}.\\

\noindent{\bf Theorem} {\it Let $M$ be a Riemannian manifold, $G$ a Lie group with a left invariant connection $\nabla^{G}$ such that its connection function $\alpha$ is skew-symmetric, $\omega_{G}$ the Maurer-Cartan form on $G$ and $F:M \rightarrow G$ a smooth map. Then $F$ is harmonic if and only if
\[
 d^{*}F^{*}\omega_{G} = 0,
\]
where $d^{*}$ is the co-differential operator on $M$.
}\\

The proof of this Theorem is based in a stochastic tool: the stochastic logarithm. It was introduced, in \cite{hawkim}, by  M. Hawkim-Dowek and D. L\'epingle. Being $X_{t}$ a semimartingale with valued in $G$, the stochastic logarithm, denoted by $\log X_{t}$, is a semimartingale in the Lie algebra $\mathfrak{g}$. 

The key of proof of Theorem above is the characterization of martingales and Brownian motions in terms of stochastic logarithm. In fact, if we take a left invariant connection on $G$ with skew-symmetric connection function $\alpha$ or a bi-invariant metric $k$, we have the following:\\

\noindent{\bf Theorem: }{\it
{\bf (i)}{\it A $G$-valued semimartingale $X_{t}$ is a $\nabla^{G}$-martingale if and only if $\log X_{t}$ is a $\nabla^{\mathfrak{g}}$-martingale, where $\nabla^{\mathfrak{g}}$ is the connection on $\mathfrak{g}$ given by $\alpha$.}\\
{\bf (ii)} {\it A semimartingale $B_{t}$ is a $k$-Brownian motion if and only in $\log B_{t}$ is a $<,>$-Brownian motion, where $<,>$ is the scalar product in $\mathfrak{g}$ associated to $k$.}}\\

This paper is organized as follow: in section 2 we give a brief exposition of stochastic calculus on manifold. In section 3 our main results are stated and proved.

\section{Preliminaries}

We begin by recalling some fundamental facts on stochastic calculus on manifolds, we shall use freely concepts and notations of M. Emery \cite{emery1} and  P. Protter \cite{protter}. In \cite{catuogno}, we find a complete survey of the stochastic properties in this section.

Let $(\Omega, (\mathcal{F}_t),\mathbb{P})$ be a filtered probability space with usual hypothesis (see for instance \cite{emery1}). In this work we mean smooth as $C^{\infty}$.

\begin{definition}
Let $M$ be a differential manifold and $X_{t}$ a continuos stochastic process with values in $M$. We call $X_{t}$ a semimartingale if, for all $f$ smooth function, $f(X_{t})$ is a real semimartingale.
\end{definition}

Let $M$ be a smooth manifold with a connection $\nabla^{M}$, $X_{t}$ a semimartingale with values in $M$, $\theta$ a section of $TM^*$ and $b$ a section of $T^{(2,0)}M$. The Stratonovich integral of $\theta$ along $X_{t}$ is denoted by $\int_0^t\theta \delta X_{s}$, the It\^o integral of $\theta$ along $X$ by $\int_0^t\theta d^{\nabla^{M}} X_{s}$. Let $(U,x_{1},\ldots,x_{n})$ be a local coordiante system on $M$. Then in $U$ we can write $b = b_{ij} dx^{i} \otimes dx^{j}$, where $b_{ij}$ are smooth functions on $U$. We define the quadratic integral of $b$ along of $X_{t}$, locally, by 
\[
\int_0^{t}b\;(dX,dX)_{s} = \int_{0}^{t} b_{ij}(X_{s}) d[X^{i}_{s},X_{s}^{j}],
\]
where $X^{i} = x_{i} \circ X$, for $i=1, \ldots n$.

\begin{definition}
Let $M$ be a smooth manifold with a connection $\nabla^{M}$. A semimartingale $X$ with values in $M$ is called a $\nabla^{M}$-martingale if $\int \theta \;d^{\nabla^{M}} X$ is a real local martingale for all $\theta \in \Gamma(TM^*)$. 
\end{definition}

\begin{definition}
Let $M$ be a Riemannian manifold with a metric $g$. Let $B$ be a semimartingale with values in $M$. We say that $B$ is a $g$-Brownian motion in $M$ if $B$ is a $\nabla^{g}$-martingale, being $\nabla^{g}$ the Levi-Civita connection of $g$, and for any section $b$ of $T^{(2,0)}M$ we have 
\begin{equation}\label{Brownian}
\int_0^tb(dB,dB)_{s}=\int_0^t \mathrm{tr}\,b(B_s)ds.
\end{equation}
\end{definition}

Following, we state the stochastic tools that are necesary to establish our main results. Firstly, we observed that
\[
\int_0^tb\;(dX,dX)_{s}=\int_0^tb^s\;(dX,dX)_{s},
\]
where 
$b^s$ is the symmetric part of $b$.

Let $M$ be a smooth manifold with a connection $\nabla^{M}$ and $\theta$ a section of $TM^*$. We have the Stratonovich-It\^o formula of conversion
\begin{equation}\label{conversion}
\int_0^t\theta \delta X_{s} = \int_0^t\theta d^{\nabla^{M}} X_{s}+\frac{1}{2}\int_0^t\nabla^{M}\theta\;(dX,dX)_{s}.
\end{equation}

When $(M,g)$ is a Riemannian manifold and $B_{t}$ is a $g$-Brownian motion in $M$ we deduce from (\ref{Brownian}) and (\ref{conversion}) the Manabe's formula:
\begin{equation}\label{manabe}
\int_0^t\theta \delta B_{s} = \int_0^t\theta d^{\nabla^{M}} B_{s}+\frac{1}{2}\int_0^t d^{*}\theta(B_{s})ds,
\end{equation}
where $d^{*}$ is the co-differential on $M$.

Let $M$ and $N$ be manifolds, $\theta$ be  a section of $TN^*$, $b$ be a section of $T^{(2,0)}N$ and $F:M\rightarrow N$ be a smooth map. For a semimartingale $X_{t}$ in $M$, we have the following It\^ o formulas for Stratonovich and quadratic integrals:
\begin{equation}\label{stratonovich-ito}
\int_0^t\theta \;\delta F(X)=\int_0^tF^*\theta \;\delta X
\end{equation}
and
\begin{equation}\label{quadratic-ito}
\int_0^tb\;(dF(X),dF(X))=\int_0^tF^*b\;(dX,dX).
\end{equation}

Let $M$ and $N$ be smooth manifolds endowed with connections $\nabla^{M}$ and $\nabla^{N}$, respectively. Let $F:M \rightarrow N$ be a smooth map and $F^{-1}(TN)$ the induced bundle. We denote by $\nabla^{N'}$ the unique connection on $F^{-1}(TN)$ induced by $\nabla^{N}$ (see for example Proposition I.3.1 in \cite{nore}). The bilinear mapping $\beta_{F}: TM \times TM \rightarrow TN$ defined by
\begin{equation}\label{secondform}
 \beta_{F}(X,Y) = \nabla^{N'}_{X}F_{*}(Y) - F_{*}(\nabla^{M}_{X}Y)
\end{equation}
is called the second fundamental form of $F$ (see for example definition I.4.1.1 in \cite{nore}). $F$ is said affine map if $\beta_{F}$ is null.

When $(M,g)$ is a Riemannian manifold, we define the tension field $\tau_{F}$ of $F$ by $\tau_{F} = \mathrm{tr}\, \beta_{F}$. We call $F$ a harmonic map if $\tau_{F} \equiv 0$. We observe that $N$ is not necessarily a Riemannian manifold to define harmonic map. But this definition is an extension of one gives by energy functional.

Let $M$ and $N$ be smooth manifold with connections $\nabla^{M}$ and $\nabla^{N}$. The It\^o geometric formula is given by:
\begin{equation}\label{itoformula}
 \int_{0}^{t} \theta d^{\nabla^{N}}F(X_{s}) = \int_{0}^{t} F^{*}\theta d^{\nabla^{M}}X_{s} + \frac{1}{2} \int_{0}^{t} \beta_{F}^{*}\theta(dX,dX)_{s}.
\end{equation}
If $(M,g)$ is  Riemannian manifold and if $B_{t}$ is a $g$-Brownian motion in $M$, then, from It\^o geometric formula and (\ref{Brownian}) we deduce that
\begin{equation}\label{tensionformula}
 \int_{0}^{t} \theta d^{\nabla^{N}}F(B_{s}) = \int_{0}^{t} F^{*}\theta d^{\nabla^{M}}B_{s} + \frac{1}{2} \int_{0}^{t} \tau_{F}^{*}\theta(B_{s})ds.
\end{equation}

From It\^o geometric formula and Doob-Meyer decomposition from real semimartingales we deduce the following stochastic characterizations for affine a harmonic maps:\\
 {\bf(i)} $F$ is an affine map if and only if it sends $\nabla^{M}$-martingales to \linebreak $\nabla^{N}$-martingales.\\
 {\bf(ii)} If $(M,g)$ is a Riemmanian manifold, then $F:M \rightarrow N$ is a harmonic map if and only if it sends $g$-Brownian motions to $\nabla^{N}$-martingales.

\section{Pluzhnikov's theorem, Brownian motions and martingales}

Let $G$ be a Lie group and $\mathfrak{g}$ its Lie algebra. Let us denote by $\omega_{G}$ the Maurer-Cartan form on $G$. Let $X_{t}$ be a semimartingale in $G$. The stochastic logarithm of the semimartingale $X_{t}$ (with $X_0 = e$) is the semimartingale, in the Lie algebra $\mathfrak{g}$, given by
\[
 \log X_{t} = \int_{0}^{t} \omega_{G} \delta X_{s}.
\]

For the convenience of the reader we repeat the following two results from \cite{catuogno2}, thus making our exposition self-contained.

\begin{lemma}\label{le1}
Let $G$ and $H$ be two Lie groups. If $\varphi:G \rightarrow H$ is a homomorphism then
\[
\varphi^{*}\omega_{H} = \varphi_{*}\omega_{G},
\]
where $\omega_{G}$ and $\omega_{H}$ be Maurer-Cartan form on $G$ and $H$, respectively.
\end{lemma}
\begin{proof}
Once $\varphi(L_{g^{-1}}(h))=L_{\varphi(g)^{-1}}(\varphi(h))$, chain rule implies that
\[
L_{\varphi(g)^{-1}*}(\varphi_{*}(v)) =\varphi_{*}(L_{g^{-1}*}(v)). 
\]
\end{proof}

\begin{proposition}\label{prop1}
Let $G$ and $H$ be two Lie groups and $\varphi:G \rightarrow H$ be a homomorphism of Lie groups. If $X_{t}$ is a $G$-valued semimartingale then
\[
\log \varphi(X_{t}) = \varphi_{*}\log X_{t}.
\]
\end{proposition}
\begin{proof}
Let $\omega_{G}$ and $\omega_{H}$ be the Maurer-Cartan form on $G$ and $H$, respectively. From (\ref{stratonovich-ito}) we see that  $\log \varphi(X_{t})  = \int_{0}^{t} \varphi^{*}\omega_{H}\delta X_{s}$. Applying Lemma \ref{le1} we obtain $\log \varphi(X_{t}) = \int_{0}^{t} \varphi_{*}\omega_{G}\delta X_{s}$.
Thus, $\log \varphi(X_{t}) = \varphi_{*}\log X_{t}$.
\eop
\end{proof}

In \cite{nomizu}, K. Nomizu proved the existence of correspondence between left invariant connections $\nabla^{G}$ on $G$ and bilinear applications $\alpha:\mathfrak{g}\otimes  \mathfrak{g} \rightarrow \mathfrak{g}$, which is  given by $\nabla^{G}_{X}Y = \alpha(X,Y)$ for all $X,Y \in \mathfrak{g}$. The bilinear application $\alpha$ is called the connection function associated to $\nabla^{G}$.

\begin{proposition}
For every bilinear application $\alpha:\mathfrak{g}\otimes  \mathfrak{g} \rightarrow \mathfrak{g}$ there exists only one connection $\nabla^{\mathfrak{g}}$ associated to $\alpha$.
\end{proposition}
\begin{proof}
 Let $X,Y$ be a vector fields in $\mathfrak{g}$. We define
\[
 \nabla^{\mathfrak{g}}_{X}Y = \alpha(X,Y),
\]
and
\[
 \nabla^{\mathfrak{g}}_{fX}Y = f\alpha(X,Y) \textrm{ and } \nabla^{\mathfrak{g}}_{X}fY = X(f)Y + f\alpha(X,Y),
\]
for $f$ smooth function on $\mathfrak{g}$. It is clear that $\nabla^{\mathfrak{g}}$ is a connection. Conversely, let $\nabla^{\mathfrak{g}}$ be a connection on $\mathfrak{g}$. Then it is sufficient to define $\alpha:\mathfrak{g}\otimes  \mathfrak{g} \rightarrow \mathfrak{g}$ as
\[
 \alpha(X,Y) = \nabla^{\mathfrak{g}}_{X}Y.
\]
It is obvious that $\alpha$ is bilinear.
\eop
\end{proof}

From now on we only work with skew-symmetric bilinear application $\alpha:\mathfrak{g}\otimes  \mathfrak{g} \rightarrow \mathfrak{g}$, and we call the associated connections $\nabla^{G}$ and $\nabla^{\mathfrak{g}}$ to $\alpha$ the skew-symmetric connections.

\begin{lemma}\label{le2}
Let $\nabla^{G}$ be a left invariant connection on $G$ and $\nabla^{\mathfrak{g}}$ a connection on $\mathfrak{g}$ such that its connection function $\alpha$ is skew-symmetric. 
\begin{enumerate}
 \item If $\theta$ is a left-invariant 1-form on $G$, then the symmetric part of $\nabla^{G} \theta$ is null.
 \item If $\theta$ is a 1-form in $\mathfrak{g}^{*}$, then the symmetric part of $\nabla^{\mathfrak{g}}\theta$ is null.
\end{enumerate}
\end{lemma}
\begin{proof}
\noindent {\it 1.} We first observe that $\nabla\theta$ is a tensor, so it is sufficiente to proof for $X, Y \in \mathfrak{g}$. Let us denote $S\nabla\theta$ the symmetric part of $\nabla\theta$. By definition of dual connection,
\begin{eqnarray*}
S\nabla \theta(X,Y)(g) 
& =  & \frac{1}{2}(X\theta(Y) + Y\theta(X) - \theta(\nabla_{X}Y +\nabla_{Y}X)(g))\\
& =  & - \frac{1}{2}\theta(\alpha(X,Y) + \alpha(Y,X))(g)
\end{eqnarray*}
Since $\alpha$ is skew-symmetric, $S\nabla \theta(X,Y)= 0$ 

\noindent {\it 2.} The proof is similar to item 1.
\eop
\end{proof}

We now prove a characterization of martingales with values in $G$ through association with martingales with values in $\mathfrak{g}$.

\begin{theorem}\label{teo1}
Let $G$ be a Lie group with a left invariant connection $\nabla^{G}$ and $\nabla^{\mathfrak{g}}$ a connection on Lie algebra $\mathfrak{g}$ such that its connection function $\alpha$ is skew-symmetric. Let $M_{t}$ be a $G$-valued semimartingale. Then $M_{t}$ is a $\nabla^{G}$-martingale if and only if $\log M_{t}$ is a $\nabla^{\mathfrak{g}}$-martingale.
\end{theorem}
\begin{proof}
We first suppose that $M_{t}$ is a $\nabla^{G}$-martingale. By definition of stochastic logarithm, 
\[
 \log M_{t} = \int_{0}^{t} \omega_{G} \delta M_{s}.
\]
Applying the formula of conversion (\ref{conversion}) we obtain
\[
 \log M_{t}  =  \int_{0}^{t} \omega_{G} d^{\nabla^{G}}M_{s} + \frac{1}{2}\int_{0}^{t} \nabla^{G} \omega_{G}(dM,dM)_{s}.
\]
Lemma \ref{le2} now assures that $\nabla^{G} \omega_{G}(dM,dM)_{t} = 0$, because the Maurer-Cartan is a left-invariant form. Thus
\[
 \log M_{t}  =  \int_{0}^{t} \omega_{G} d^{\nabla^{G}}M_{s}.
\] 
We observe that $\log M_{t}$ is a local martingale.  For $\theta \in \mathfrak{g}^{*}$ we have that
\[
 \int_{0}^{t} \theta \delta \log M_{t} =  \int_{0}^{t} \theta\omega_{G} d^{\nabla^{G}}M_{s}.
\]
From formula of conversion (\ref{conversion}) we see that 
\[
 \int_{0}^{t} \theta d^{\nabla^{\mathfrak{g}}}\log M_{t} + \frac{1}{2}\int \nabla^{\mathfrak{g}}\theta(d\log M_{s},d\log M_{s}) =  \int_{0}^{t} \theta\omega_{G} d^{\nabla^{G}}M_{s}.
\]
Lemma \ref{le2} leads to $\int \nabla^{\mathfrak{g}}\theta(d\log M_{s},d\log M_{s})= 0$. Thus
\[
 \int_{0}^{t} \theta d^{\nabla^{\mathfrak{g}}}\log M_{s} =  \int_{0}^{t} \theta\omega_{G} d^{\nabla^{G}}M_{s}.
\]
Since $\int_{0}^{t} \theta\omega_{G} d^{\nabla^{G}}M_{s}$ is a real local martingale, we conclude that $\log M_{t}$ is a $\nabla^{\mathfrak{g}}$-martingale.

Conversely, let $\theta$ be a left invariant 1-form in $G$. Using the formula of conversion (\ref{conversion}) and Lemma \ref{le2} leads to
\[
\int_{0}^{t} \theta d^{\nabla^{G}}M_{s}  =  \int_{0}^{t} \theta \delta M_{s}.
\]
Writing $\theta_{g} = \theta_{e}\circ \omega_{G} $ we obtain 
\[
\int_{0}^{t} \theta d^{\nabla^{G}}M_{s}  =  \int_{0}^{t} \theta_{e}\omega_{G} (\delta M_{s}).
\] 
By definition of logarithm,
\[
\int_{0}^{t} \theta d^{\nabla^{G}}M_{s} =  \int_{0}^{t} \theta_{e}\delta \log M_{s}.
\]
Applying the formula of conversion (\ref{conversion}) we see that
\[
\int_{0}^{t} \theta d^{\nabla^{G}}M_{s} =  \int_{0}^{t} \theta_{e}d^{\nabla^{\mathfrak{g}}}\log M_{s}  + \frac{1}{2}\int_{0}^{t} \nabla^{\mathfrak{g}}\theta_{e}(d\log M_{s},d\log M_{s}),
\]
being $\nabla^{\mathfrak{g}}$ the connection on $\mathfrak{g}$ yielded by connection function $\alpha$. From Lemma \ref{le2} it follows that $\nabla^{\mathfrak{g}}\theta_{e}=0$. Thus
\[
\int_{0}^{t} \theta d^{\nabla^{G}}M_{s} =  \int_{0}^{t} \theta_{e}d^{\nabla^{\mathfrak{g}}}\log M_{s}.
\]
Since $\log M_{t}$ is a $\nabla^{\mathfrak{g}}$-martingale, we conclude that $M_{t}$ is a $\nabla^{G}$-martingale.
\eop
\end{proof}

The next corollary is a direct consequence of theorem above, but it is not possible to show its converse with the tools that we are using here.

\begin{corollary}\label{cor1}
Let $G$ be a Lie group with a left invariant connection $\nabla^{G}$, which has a skew-symmetric connection function $\alpha$. If  $M_{t}$ is a $\nabla^{G}$-martingale, then $\log M_{t}$ is a local martingale in $\mathfrak{g}$.
\end{corollary}

\begin{example}
 Let $\alpha$ be the connection function null. Then, from Theorem \ref{teo1} we conclude that $M_{t}$ is a $\nabla^{G}$-martingale if and only if $\log M_{t}$ is local martingale in $\mathfrak{g}$. It was first proved by M. Arnaudon in \cite{arnaudon1}.
\end{example}

We know that there exists an one-to-one association between bi-invariant metrics on Lie group $G$ and $Ad_{G}$-invariant scalar products $<,>$ on Lie algebra $\mathfrak{g}$. We will use this to give the following characterization for Brownian motion in $G$.

\begin{theorem}\label{teo2}
Let $G$ be a Lie group whit a bi-invariant metric $k$. Let $B_{t}$ be a semimartingale in $G$. Then $B_{t}$ is a Brownian motion in $G$ if and only in $\log B_{t}$ is a $<,>$-Brownian motion in $\mathfrak{g}$.
\end{theorem}
\begin{proof}
We first observe that the Levi-Civita connection associated to metric $k$ is given by 
\[
\nabla^{k}_{X}Y = \frac{1}{2}[X,Y]
\]
for all $X,Y \in \mathfrak{g}$ (see for example \cite{arvani}). 

Suppose that $B_{t}$ is a $k$-Brownian motion. From definition and Theorem \ref{teo1} we know that $\log B_{t}$ is a $\nabla^{\mathfrak{g}}$-martingale in $\mathfrak{g}$, where $\nabla^{\mathfrak{g}}$ is connection generate by $\frac{1}{2}[\cdot,\cdot ]$. It remains to prove that $\int_{0}^{t}b(d\log B,d\log B)_{s} =\int_{0}^{t}\mathrm{tr}\, (\log B_{s}) ds$, where $b$ is a bilinear form in $\mathfrak{g}$. In fact, let $(x_{1},\ldots, x_{n})$ be a global coordinates system of $\mathfrak{g}$. Thus, we can write $b = b_{ij} dx^{i} \otimes dx^{j}$, where $b_{ij}$ are smooth functions on $\mathfrak{g}$. By definition,
\begin{eqnarray*}
\int_{0}^{t} b(d\log B,d\log B)_{s} 
& = & \int_{0}^{t} b_{ij}(\log B_{s}) [\log B_{s}^{i},\log B_{s}^{j}]\\
& = & \int_{0}^{t} b_{ij}(\log B_{s}) d \int_{0}^{s} [\log B_{r}^{i},\log B_{r}^{j}]\\
& = & \int_{0}^{t} b_{ij}(\log B_{s}) d \int_{0}^{t} dx^{i}\otimes dx^{j} (d\log B_{r}, d\log B_{r})\\
& = & \int_{0}^{t} b_{ij}(\log B_{s}) d \int_{0}^{t} dx^{i}\otimes dx^{j} (L_{B_{r}^{-1}*} dB_{r}, L_{B_{r}^{-1}*}dB_{r})\\
& = & \int_{0}^{t} b_{ij}(\log B_{s}) d \int_{0}^{t} dx^{i}\circ L_{B_{r}^{-1}*} \otimes dx^{j}\circ L_{B_{r}^{-1}*}(dB_{r},dB_{r}),
\end{eqnarray*}
where we used the Theorem 3.8 of \cite{emery1} in the second and third equality. Being $B_{t}$ a Brownian motion, 
\[
 \int_{0}^{t} b(d\log B,d\log B)_{s} = \int_{0}^{t} b_{ij}(\log B_{s}) d\int_{0}^{s} \mathrm{tr}\, (dx^{i}\circ L_{B_{r}^{-1}*} \otimes dx^{j}\circ L_{B_{r}^{-1}*})(B_{r}) dr.
\]
As $k$ is a bi-invariant metric we have 
\begin{eqnarray*}
 \int_{0}^{t} b(d\log B,d\log B)_{s} & = & \int_{0}^{t} b_{ij}(\log B_{s}) d\int_{0}^{s} k^{ij}(\log B_{r})dr \\
& = &  \int_{0}^{t} b_{ij}(\log B_{s}) k^{ij}(\log B_{s}) ds = \int_{0}^{t} \mathrm{tr}\, b(\log B_{s})ds,
\end{eqnarray*}
where $k^{ij}$ are the coeficients of inverse matrix $(<,>_{ij})$. Therefore $\log B_{t}$ is a $<,>$-Brownian motion in $\mathfrak{g}$.

Conversely, suppose that $\log B_{t}$ is a $<,>$-Brownian motion in $\mathfrak{g}$. It remains to prove (\ref{Brownian}). For each $b \in T^{(0,2)}(G)$,
\[
\int_{0}^{t} b(dB, dB)_{s} = \int_{0}^{t} b(L_{B_{s}*}L_{B_{s}^{-1}*} dB_{s},L_{B_{s}*}L_{B_{s}^{-1}*} dB_{s}).
\]
By definition of logarithm,
\[
\int_{0}^{t} b(dB, dB)_{s} = \int_{0}^{t} L_{B_{s}}^{*}b(d\log B,d\log B)_{s}.
\]
As $\log B_{t}$ is a $<,>$-Brownian motion we have
\[
\int_{0}^{t} b(dB, dB)_{s} = \int_{0}^{t} tr(L_{B_{s}}^{*}b)(\log B_{s})ds.
\]
Being $k$ bi-invariant metric, we have 
\[
\int_{0}^{t} b(dB, dB)_{s} = \int_{0}^{t} \mathrm{tr}\, b (B_{s})ds.
\]
Thus $B_{t}$ is a $k$-Brownian motion in $G$.
\eop
\end{proof}

As consequence of Theorem above, every $k$-Brownian motion in $G$ yields a standart Brownian motion in $\mathfrak{g}$, but, as Corollary \ref{cor1}, we can not show the converse with these arguments.

\begin{corollary}
Let $G$ be a Lie group whit a bi-invariant metric $k$. If $B_{t}$ is a $k$-Brownian motion in $G$, then $\log B_{t}$ is a Brownian motion in $\mathfrak{g}$.
\end{corollary}
\begin{proof}
It follows from Corollary \ref{cor1} that if $B_{t}$ is a  $\nabla^{k}$-martingale, where $\nabla^{k}$ is the Levi-Civita connection associated to metric $k$, then $\log B_{t}$ is a local martingale in $\mathfrak{g}$.  By Levi's characterization of $n$-dimensional Brownian motion remains to prove that $[\log B_{t}^{i},\log B_{t}^{j}]_{t}=\delta^{i}_{j}t$ (see \cite{protter} for more details). In fact, we make
\[
 [\log B_{t}^{i},\log B_{t}^{j}] = \int_{0}^{t} d[\log B_{s}^{i},\log B_{s}^{j}]
\]
and we apply the first part of the demonstration of Theorem \ref{teo2} to conclude the proof.
\eop
\end{proof}

As an application of Theorems \ref{teo1} and \ref{teo2} we prove the useful results. Someone will be able to show the next Proposition whit geometric arguments.

\begin{proposition}
Let $G$ be a Lie group and $H$ a Lie group with a left invariant connection $\nabla^{H}$, which has a skew-simmetric connection function $\alpha$ and $\varphi:H \rightarrow G$ an homorphism of Lie groups. We have the following assertions:
\begin{description}
 \item[(i)] If $G$ has a left invariant connection $\nabla^{G}$ such that its connection function is skew-symmetric and if $\varphi_{e*}$ commutes with $\alpha$, then every homomorphism $\varphi:H \rightarrow G$ is an affine map.
 \item[(ii)] If $G$ has a bi-invariant metric $k$ and if $\varphi_{e*}$ commutes with $\alpha$, then every homomorphism $\varphi:H \rightarrow G$ is a harmonic map.
\end{description}
\end{proposition}
\begin{proof}
{\bf (i)} 
Let $M_{t}$ be a $\nabla^{H}$-martingale in $H$. It is sufficient to show that $\varphi(M_{t})$ is a  $\nabla^{G}$- martingale. In fact, Theorem (\ref{teo2}) shows that $\log M_{t}$ is a $\nabla^{\mathfrak{h}}$-martingale in the Lie algebra $\mathfrak{h}$. By Proposition \ref{prop1}, 
\[
\log \varphi(M_{t})  = \varphi_{e*}(\log M_{t}).
\]
Since $\varphi_{e*}$ commute with $\alpha$, from It\^o geometric formula we deduce that  $\log \varphi(M_{t})$ is a $\nabla^{\mathfrak{g}}$-martingale in the Lie algbra $\mathfrak{g}$. Theorem \ref{teo1} shows that $\varphi(M_{t})$ is a $\nabla^{G}$-martingale.

\noindent{\bf (ii)}
Let $B_{t}$ be a $k$-Brownian motion in $G$. From stochastic characterization for harmonic maps is suficient to show that $\varphi(B_{t})$ is a $\nabla^{G}$- martingale. In fact, Theorem (\ref{teo1}) shows that $\log B_{t}$ is a $\nabla^{\mathfrak{h}}$-martingale. By Proposition \ref{prop1}, 
\[
\log \varphi(B_{t})  = \varphi_{e*}(\log B_{t}).
\]
Because $\varphi_{e*}$ commutes with $\alpha$, the It\^o formula assures that $\log \varphi(B_{t})$ is a $\nabla^{\mathfrak{g}}$-martingale. Theorem \ref{teo1} shows that $\varphi(B_{t})$ is a $\nabla^{G}$-martingale.
\end{proof}

\begin{example}
Let $G, H$ be two Lie groups. If we equippe $G$ whit a connection $\nabla^{G}_{X}Y = c_{1}[X,Y]$ for some $c_{1} \in [0,1]$ and, $X,Y \in \mathfrak{g}$, and if we endow $H$ with a connection $\nabla^{H}_{\tilde{X}}\tilde{Y} = c_{2}[\tilde{X},\tilde{Y}]$ for some $c_{2} \in [0,1]$, $\tilde{X},\tilde{Y} \in \mathfrak{h}$, then every homomorphism of Lie groups $\varphi:H \rightarrow G$ is an affine map. When $H$ has a bi-invariant metric, every homomorphism of Lie groups $\varphi:H \rightarrow G$ is a harmonic map.
\end{example}

The next Lemma is necessary in the proof the Pluzhnikov's Theorem. We observe that it is true, because we work in the Lie algebra context.

\begin{lemma}\label{le3}
Let $(M,g)$ be a Riemannian manifold, $G$ a Lie group, $\omega_{G}$ the Maurrer-Cartan form on $G$ and $F:M \rightarrow G$ a smooth map. Then 
\[
 d^{*}F^{*}\omega_{G}^{*}\theta =  \theta d^{*}F^{*}\omega_{G},
\]
for every $\theta$ 1-form on $\mathfrak{g}$, where $d^{*}$ is the co-differential operator on $M$.
\end{lemma}
\begin{proof}
From definition of co-differential $d^{*}$, for any orthonormal frame field $\{e_{1},\ldots, e_{n}\}$ on $M$, we have
\[
 d^{*}F^{*}\omega_{G}^{*}\theta = -\sum_{i=1}^{n} (\nabla^{g}_{e_{i}}F^{*}\omega_{G}^{*}\theta)(e_{i}),
\]
where $\nabla^{g}$ is the Levi-Civita connection associated to metric $g$. By definition of dual connection,
\begin{eqnarray*}
d^{*}F^{*}\omega_{G}^{*}\theta 
& = & -\sum_{i=1}^{n}( \nabla^{g}_{e_{i}}(F^{*}\omega_{G}^{*}\theta(e_{i})) - F^{*}\omega_{G}^{*}\theta(\nabla^{g}_{e_{i}}e_{i}))\\
& = & -\sum_{i=1}^{n}(e_{i}\theta(F^{*}\omega_{G}(e_{i})) - \theta(F^{*}\omega_{G}\nabla^{g}_{e_{i}}e_{i}))\\
\end{eqnarray*}
Since $\theta:\mathfrak{g} \rightarrow \mathbb{R}$ is a linear application, we obtain
\begin{eqnarray*}
d^{*}F^{*}\omega_{G}^{*}\theta 
& = & \theta(-\sum_{i=1}^{n}(\nabla^{g}_{e_{i}}(F^{*}\omega_{G}(e_{i})) - F^{*}\omega_{G}^{*}\nabla^{g}_{e_{i}}e_{i}))\\
& = & \theta(d^{*}F^{*}\omega_{G}),
\end{eqnarray*}
where we used the definition of co-differential in the last equality.
\eop
\end{proof}

Finally, we will prove a version of Pluzhnikov's Theorem (see \cite{pluzhnikov}) to skew-symmetric connections. 

\begin{theorem}
 Let $M$ be a Riemannian manifold, $G$ a Lie group with a left invariant connection $\nabla^{G}$ such that its connection function $\alpha$ is skew-symmetric, $\omega_{G}$ the Maurer-Cartan form on $G$ and $F:M \rightarrow G$ a smooth map. Then $F$ is harmonic if and only if
\[
 d^{*}F^{*}\omega_{G} = 0,
\]
where $d^{*}$ is the co-differential operator on $M$.
\end{theorem}
\begin{proof}
Suppose that $F$ is a harmonic map. From stochastic characterization for harmonic maps we have, for every $g$-Brownian motion $B_{t}$ in $M$, that $F(B_{t})$ is a $\nabla^{G}$-martingale in $G$. From Theorem \ref{teo1} we see that $\log F(B_{t})$ is a $\nabla^{\mathfrak{g}}$-martingale, where $\nabla^{\mathfrak{g}}$ is the connection given by $\alpha$ in $\mathfrak{g}$. Let $\theta$ be a 1-form on $\mathfrak{g}$. From formula of conversion (\ref{conversion}) we deduce that
\begin{eqnarray*}
 \int_{0}^{t} \theta d^{\nabla^{\mathfrak{g}}}\log F(B_{s}) \!\!\!
& = & \!\!\! \int_{0}^{t} \theta \delta \log F(B_{s}) -\frac{1}{2} \int_{0}^{t} \nabla^{\mathfrak{g}}\theta(d\log F(B_{s}),d\log F(B_{s}))\\
& = & \int_{0}^{t} \omega_{G}^{*} F^{*}\theta \delta B_{s} -\frac{1}{2} \int_{0}^{t} \nabla^{\mathfrak{g}}\theta(d\log F(B_{s}),d\log F(B_{s})),
\end{eqnarray*}
where we used the definition of stochastic logarithm and property (\ref{stratonovich-ito}) in the second equality. Because $\nabla^{\mathfrak{g}}$ is given by $\alpha$ and $\alpha$ is skew-symmetric Lemma \ref{le2} assures that
\[
 \int_{0}^{t} \theta d^{\nabla^{\mathfrak{g}}}\log F(B_{s}) = \int_{0}^{t} \omega_{G}^{*}F^{*}\theta \delta B_{s}.
\]
Manabe's formula (\ref{manabe}) now yields
\[
 \int_{0}^{t} \theta d^{\nabla{\mathfrak{g}}}\log F(B_{s}) 
 = \int_{0}^{t} F^{*}\omega_{G}^{*}\theta d^{\nabla^{g}} B_{s} + \frac{1}{2}\int_{0}^{t} d^{*}\omega_{G}^{*}F^{*}\theta (B_{s}) ds.
\]
Since $\log F(B_{t})$ is a $\nabla^{\mathfrak{g}}$-martingale, from Doob-Meyer decomposition (see for instance \cite{protter}) we deduce that
\[
\int_{0}^{t} d^{*}F^{*}\omega_{G}^{*}\theta(B_{s}) dt = 0.
\]
Since $B_{s}$ is an arbitrary $g$-Brownian motion, it follows that $d^{*}F^{*}\omega_{G}^{*}\theta = 0$, where $d^{*}$ is the co-differential operator on $M$. From Lemma \ref{le3} we see that $\theta(d^{*}F^{*}\omega_{G})$ = 0. Being $\theta$ an arbitrary 1-form on $\mathfrak{g}$, we conclude that 
\[
d^{*}F^{*}\omega_{G} = 0.
\]

Conversely, suppose that $d^{*}F^{*}\omega_{G}^{*}= 0$. We want to show, for every $g$-Brownian motion $B_{s}$ in $M$, that $F(B_{s})$ is a $\nabla^{G}$-martingale in $G$. To this end, we will show that $\log F(B_{s})$ is a $\nabla^{\mathfrak{g}}$-martingale in $\mathfrak{g}$ and we will conclude from Theorem \ref{teo1} our assertion. In fact, for $\theta \in \mathfrak{g}^{*}$ we can repeat to arguments above and to obtain
\begin{eqnarray*}
\int_{0}^{t} \theta d^{\nabla^{\mathfrak{g}}}\log F(B_{s}) 
& = & \int_{0}^{t} F^{*}\omega^{*}\theta d^{\nabla^{g}} B_{s} + \frac{1}{2}\int_{0}^{t} d^{*}\omega^{*}F^{*}\theta(B_{s}) ds.
\end{eqnarray*}
From Lemma \ref{le3} and the hypothesis we get
\[
 \int_{0}^{t} \theta d^{\nabla^{\mathfrak{g}}}\log F(B_{s}) 
 = \int_{0}^{t} F^{*}\omega^{*}\theta d^{\nabla^{g}} B_{s} .
\]
Because $B_{t}$ is a $g$-Brownian motion, by definition, $\int_{0}^{t} \theta d^{\nabla^{\mathfrak{g}}}\log F(B_{s})$ is a local martingale. Furthermore,  $\log F(B_{s})$ is a $\nabla^{\mathfrak{g}}$-martingale in $\mathfrak{g}$, and the proof follows.
\eop
\end{proof}


\begin{thebibliography}{20}


\bibitem{arvani} Arvanitoyeorgos, A. {\it An introduction to Lie groups and the geometry of homogeneous spaces.}  Student Mathematical Library, 22. American Mathematical Society, Providence, RI, 2003.

\bibitem{arnaudon1} Arnaudon, M., {\it Conexions et Martingales dans les Groupes de Lie}, Lecture Notes in Mathematics, 1526, 1992, p. 146 - 155.







\bibitem{catuogno} Catugno P., {\it A geometric It\^o formula}, Matem\'atica Contemporr\^anea, Vol 33, 85-99.

\bibitem{catuogno2} Catuogno P., Ruffino P., {\it Stochastic Exponential in Lie Groups and its Applications} IMECC-Unicamp, Research reporter 07/2003.





\bibitem{hawkim} Hawkim-Dowek, M., and L\'epingle, D., {\it L'exponentielle Stochastique de Groupes de Lie},Lectures Notes in Mathematics, 1204, 1986, p. 352-374.


\bibitem{emery1} Emery, M., {\it Stochastic Calculus in Manifolds}, Springer, Berlin 1989.

















\bibitem{nomizu} Nomizu, K., {\it Invariant affine connections on homogeneous spaces}, Amer. J. Math.  76,  (1954). 33--65.

\bibitem{nore} Nore, Th\'er\`ese, {\it Second fundamental form of a map.} Ann. Mat. Pura Appl. (4) 146 (1987), 281--310. 





\bibitem{pluzhnikov} Pluzhnikov, A. I. {\it Some properties of harmonic mappings in the case of spheres and Lie groups.} (Russian)  Dokl. Akad. Nauk SSSR  268  (1983),  no. 6, 1300--1302.

\bibitem{protter} Protter, P., {\it Stochastic integration and differential equations. A new approach.} Applications of Mathematics (New York), 21. Springer-Verlag, Berlin, 1990.















\end{thebibliography}
\end{document}